\documentclass[12pt,oneside]{amsart}
\usepackage[latin9]{inputenc}
\usepackage{amsthm}
\usepackage{amstext}
\usepackage{amssymb}
\usepackage{graphicx}

\makeatletter

\providecommand{\tabularnewline}{\\}

\numberwithin{equation}{section}
\numberwithin{figure}{section}

\title{}
\def\qed{\quad\vrule height4.17pt width4.17pt depth0pt}

\usepackage{amsfonts}\usepackage{enumerate}

\newtheorem{theorem}{Theorem}
\theoremstyle{plain}

\newtheorem{corollary}{Corollary}

\newtheorem{definition}{Definition}
\newtheorem{example}{Example}

\numberwithin{equation}{section}

\makeatother

\begin{document}

\title{Some general properties of LAD and\linebreak{}
RAD AG-groupoids}

\author{M. $\mbox{Rasha\ensuremath{d^{A}}}$}

\author{I. $\mbox{Ahma\ensuremath{d^{A,*}}}$}

\email{rashad@uom.edu.pk}

\email{iahmaad@hotmail.com}

\address{A: Department of Mathematics, University of Malakand, Chakdara, Pakistan.}

\author{M. Shah}

\email{shahmaths\_problem@hotmail.com}

\address{Government Post Graduate College Mardan, Pakistan.}

\keywords{AG-groupoid, locally associative AG-groupoid, paramedial AG-groupoid,
LAD-groupoid, AG$^{**}$-groupoid, congruences.\\
{*}Corresponding author}

\maketitle
\noindent \begin{center}
2010 Mathematics Subject Classifi{}cation: \textbf{20N02}
\par\end{center}
\begin{abstract}
A groupoid that satisfies the left invertive law: $ab\cdot c=cb\cdot a$
is called an AG-groupoid. We extend the concept of left abelian distributive
groupoid (LAD) and right abelian distributive groupoid (RAD) to introduce
new subclasses of AG-groupoid, left abelian distributive AG-groupoid
and right abelian distributive AG-groupoid. We give their enumeration
up to order 6 and find some basic relations of these new classes with
other known subclasses of AG-groupoids and other relevant algebraic
structures. We establish a method to test an arbitrary AG-groupoid
for these classes. 
\end{abstract}

\section{\textbf{Introduction}}

An AG-groupoid is a generalization of commutative semigroup, in which the left invertive law: $(ab)c=(cb)a$ holds \cite{KM}. A groupoid $S$ is called left (resp. right) abelian distributive groupoids if
it satisfies $a\cdot bc=ab\cdot ca\,(\textrm{resp. }ab\cdot c=ca\cdot bc)$
\cite{DK}. Here we extend these concepts to left (right) abelian
distributive AG-groupoids. We prove their existence by providing non-associative
examples of various finite orders. We also prove their relations with
some of the already known subclasses\cite{DSV,MIM,key-7,RAAS,IMM,leftT-2-1}
of AG-groupoid and with other algebraic structures. AG-groupoids are
enumerated up to order $6$ in \cite{DSV} using GAP~\cite{GAP}
by one of our author. We also use the same technique to enumerate
our new subclasses of AG-groupoids. Table 1, contains the counting
of these new subclasses of AG-groupoids. An AG-groupoid S is called
monoid if it contains a unique left identiy. Every monoid satisfies
the paramedial property. It is also interesting to note that if S
possesses the right identity element then it becomes a semigroup.
$S$ is called medial if it satisfies the identity, $ab.cd=ac.bd$.
It is is easy to prove that every AG-groupoid is medial. AG-groupoids
have a variety of applications in flocks and geometry\cite{KM,SM}.
In the following we give some defintions of AG-groupoids with their
identies that will be used in the rest of this article.

An Abel-Grassmann groupoid $S$ \cite{KM}, abbreviated as an AG-groupoid,
is a groupoid that satisfies the left invertive law, 
\begin{eqnarray}
ab\cdot c & = & cb\cdot a\,\forall a,b,c\in S\label{eq:a}
\end{eqnarray}
An AG-groupoid $S$ is called ---

\noindent \begin{enumerate}[(i)]

\item --- right commutative AG-groupoid, if $a\cdot bc=a\cdot cb$
\cite{MIM}. 

\item --- self-dual AG-groupoid, if $a\cdot bc=c\cdot ba$ \cite{SM}
.

\item --- left distributive (LD) AG-groupoid, if $a\cdot bc=ab\cdot ac$
\cite{SM}. 

\item --- right distributive (RD) AG-groupoid, if $ab\cdot c=ac\cdot bc$
\cite{SM}.

\item --- an AG$^{**}$-groupoid if it satisfies the identity $a(bc)=b(ac)$
\cite{SM}.

\item --- flexible AG-groupoid if it satisfies the identity $a(ba)=(ab)a$\cite{SM}.

\item --- paramedial AG-groupoid if it satisfies the identity $ab.cd=db.ca$
\cite{RJT}.

\item --- medial AG-groupoid if it satisfies the identity $ab.cd=ab.cd$
\cite{RJT}.

\item --- right commute AG-groupoid if it satisfies the identity$a\cdot bc=a\cdot cb$.

\noindent \end{enumerate}

\begin{table}[h]
\begin{centering}
\begin{tabular}{|l|c|c|c|c|}
\hline 
Order & ~~~3~~~ & ~~~4~~~  & 5 & 6\tabularnewline
\hline 
\hline 
Total (AG-groupoids) & 20 & 331 & 31913 & 40104513\tabularnewline
\hline 
Non associative RAD AG-groupoids & 6 & 175 & 21186 & 34539858\tabularnewline
\hline 
Non associative LAD AG-groupoids & 0 & 0 & 27 & 1106\tabularnewline
\hline 
Non associative AD AG-groupoids & 0 & 0 & 0 & 0\tabularnewline
\hline 
\end{tabular}
\par\end{centering}

\bigskip{}

\caption{Enumeration of RAD \& LAD AG-groupoids up to order $6$.}
\end{table}

\section{\textbf{Left Abelian Distributive AG-groupoids}}

\begin {definition} An AG-groupoid $S$ is called left abelian distributive
AG-groupoid, denoted by LAD if $\forall a,b,c\in S,$
\begin{eqnarray}
a\cdot bc & = & ab\cdot ca\label{eq:a-1}
\end{eqnarray}

\noindent \end {definition}

\noindent \begin {example} Let $S=\{1,2,3,4\}$ with the following
table. Then it is easy to verify that $S$ is LAD AG-groupoid. \end {example}

\begin{center}
\begin{tabular}{c|cccc}
{*} & 1 & 2 & 3 & 4\tabularnewline
\hline 
1 & 1 & 1 & 1 & 1\tabularnewline
2 & 1 & 1 & 1 & 1\tabularnewline
3 & 1 & 1 & 1 & 2\tabularnewline
4 & 1 & 1 & 1 & 3\tabularnewline
\end{tabular}
\par\end{center}

\section{\textbf{Left Abelian Distributive AG-test}}

In this section we discuss a procedure of \cite{PN} that whether
an AG-groupoid $(G,\cdot)$ is LAD-AG-groupoid or not, for this we
define the following binary operation
\begin{eqnarray}
a\circledcirc b & = & ab\cdot xa\label{eq:c}\\
a\circ b & = & a\cdot bx
\end{eqnarray}
The law $a\cdot bx=ab\cdot xa$ is satisfied if,
\begin{eqnarray}
a\circledcirc b & = & a\circ b\label{eq:d}
\end{eqnarray}

To construct the table of operation $``\circledcirc"$ for any fixed
$x\in G.$ We rewrite $x$-row of the $``\cdot"$ table as an index
row for the new extended table. Multiply corresponding elements of
each row of the original table by the corresponding element of the
each index row for the extended table. Similarly the table of the
operation $``\circ"$ for any fixed $x\in G$ is obtained by rewriting
the $x$-column of $``\cdot"$ table as an index colum and multiplying
it by the elements of the index row from the left. If the tables obtained
for the operation $``\circledcirc"$ and $``\circ"$ coincides for
all $x\in G,$ then the equation \ref{eq:a-1} holds and the AG-groupoid
is an LAD-AG-groupoid in that case. We illustrate the procedure with
the following example.

$\newline$

\noindent \begin{example} Verify the AG-groupoid $G$ with the following
table for LAD- AG-groupoid.\end{example}\medskip{}

\noindent \begin{center}
\begin{tabular}{|c|ccccc|}
\hline 
$\cdot$ & 1 & 2 & 3 & 4 & 5\tabularnewline
\hline 
1 & 1 & 1 & 1 & 1 & 1\tabularnewline
2 & 1 & 2 & 2 & 2 & 2\tabularnewline
3 & 1 & 2 & 2 & 2 & 2\tabularnewline
4 & 1 & 2 & 2 & 2 & 3\tabularnewline
5 & 1 & 2 & 2 & 2 & 4\tabularnewline
\hline 
\end{tabular}
\par\end{center}

Extend the above table in the way as described above we get the following
table, where the tables on right to the original table is for the
operation $``\circledcirc"$ and the tables that lies downwards of
the original tables are constructed for the operation $``\circ"$.

\begin{figure}
\noindent \centering{}\includegraphics[width=13cm]{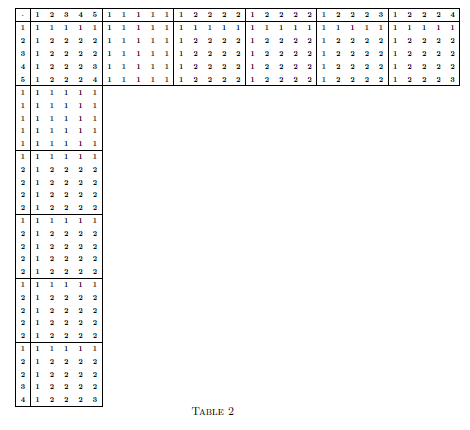}
\end{figure}

\noindent \begin{flushleft}
It is clear from the extended table that the downward tables and
the tables on the right coincide so $G$ is LAD-AG-groupoid.
\par\end{flushleft}

\section{Characterization of LAD-AG-groupoid}

\noindent \begin{theorem} \label{Lem1}Let S be an LAD-AG-groupoid.
then the following hold.

\noindent $S$ is RC-AG-groupoid

\begin{enumerate}[(i)]

\item $S$ is right commutative (RC) AG-groupoid;

\item $S$ is self-dual AG-groupoid;

\item $S$ is AG$^{**}$-groupoid;

\item $S$ is left distributive AG-groupoid.

\end{enumerate}\end{theorem}

\noindent \emph{Proof.} Let $S$ be an LAD-AG-groupoid and $a,b,c\in S.$
Then

\begin{enumerate}[(i)]

\item Using the Identity \ref{eq:a-1} and the medial law, we have
$a\cdot bc=ab\cdot ca=ac\cdot ba=a\cdot cb$. Thus $S$ is RC-AG-groupoid.

\item \label{lem 2-1} Using the medial law and the Identity \ref{eq:a-1},
and Lemma \ref{Lem1} (i), we get, 
\begin{eqnarray*}
a\cdot bc & = & ab\cdot ca=(ab\cdot c)(a\cdot ab)=(ab\cdot c)(a\cdot ba)=(ab\cdot a)(c\cdot ba)=\\
 & = & (ab\cdot a)(cb\cdot ac)=(ab\cdot cb)(a\cdot ac)=(ab\cdot cb)(aa\cdot ca)=\\
 & = & (ab\cdot aa)(cb\cdot ca)=(ab\cdot aa)(cb\cdot ac)=(a\cdot ba)(c\cdot ba)=\\
 & = & (ac)(ba\cdot ba)=(ac)(ba\cdot ab)=(ac)(b\cdot aa)=(ac)(aa\cdot b)=\\
 & = & (ac)(ba\cdot a)=c\cdot ba\Rightarrow a\cdot bc=c\cdot ba.
\end{eqnarray*}
Hence $S$ is self-dual AG-groupoid.

\item Using the Identity \ref{eq:a-1}, medial law and Theorem (\ref{Lem1}
(i)), we get

\begin{center}
$a\cdot bc=a\cdot cb=b\cdot ca=bc\cdot ab=ba\cdot cb=b\cdot ac.$
\par\end{center}

\noindent Hence $S$ is AG$^{**}$-groupoid.

\noindent \item The Identity \ref{eq:a-1} and medial law imply that
\begin{eqnarray*}
a\cdot bc & = & ab\cdot ca=(ab\cdot c)(a\cdot ab)\\
 & = & (ab\cdot a)(c\cdot ab)=ab\cdot ac\\
\Rightarrow a\cdot bc & = & ab\cdot ac.
\end{eqnarray*}
Hence $S$ is left distributive AG-groupoid.\qed 

\end{enumerate}

\noindent \begin {corollary}\label{cor 1} Every LAD-AG-groupoid
is paramedial AG-groupoid and hence is a left nuclear square AG-groupoid.\end {corollary}

It is worth mentioning that the concepts of LD and LAD for AG-groupoids
are different. To this end we give an example of LD-AG-groupoid which
is not LAD-AG-groupoid.

\noindent \begin{flushleft}
\begin {example}Let $S=\{1,2,3,4\}$ with the following table. Then
one can easily verify that S is an LD-AG-groupoid. Since $1\cdot24\neq12\cdot41$
thus $S$ is not an LAD-AG-groupoid. \end {example}
\par\end{flushleft}

\begin{center}
\begin{tabular}{c|cccc}
{*} & 1 & 2 & 3 & 4\tabularnewline
\hline 
1 & 1 & 3 & 4 & 2\tabularnewline
2 & 4 & 2 & 1 & 3\tabularnewline
3 & 2 & 4 & 3 & 1\tabularnewline
4 & 3 & 1 & 2 & 4\tabularnewline
\end{tabular}
\par\end{center}

\section{\textbf{RAD-AG-groupoid}}

\begin {definition} An AG-groupoid $S$ is called right abelian distributive
(or shortly RAD) AG-groupoid if $\forall a,b,c\in S,$
\begin{eqnarray}
ab\cdot c & = & ca\cdot bc\label{eq:b-1}
\end{eqnarray}

\noindent \end {definition}

\noindent \begin {example} Let $S=\{a,b,c\}$ with the following
table then one can easily verify that S is an RAD AG-groupoid.

\begin{center}
\begin{tabular}{c|ccc}
{*} & a & b & c\tabularnewline
\hline 
a & a & a & a\tabularnewline
b & a & a & a\tabularnewline
c & a & b & a\tabularnewline
\end{tabular}
\par\end{center}

\end {example}

\section{\textbf{RAD-AG-Test}}

In this section we discuss the procedure that how to check an arbitrary
AG-groupoid $(G,\cdot)$ for an RAD-AG-groupoid, for this we define
the following binary operation
\begin{eqnarray}
a\circ b & = & xa\cdot bx\label{eq:e}\\
a\diamondsuit b & = & ab\cdot x\label{eq:f}
\end{eqnarray}
The law $ab\cdot x=xa\cdot bx$ is satisfied if;
\begin{eqnarray}
a\circ b & = & a\diamondsuit b\label{eq:g}
\end{eqnarray}
To construct table of operation $``\circ"$ for any fixed $x\in G,$
rewriting $x$-row of the $``\cdot"$ table as an index row of the
new extended table and multiply its elements by the elements of the
$x$-column of $``\cdot"$ table turn by turn to list the rows of
$``\circ"$ tables. 

Similarly the tables of the operation $``\diamondsuit"$ for any fixed
$x\in G$ is constructed by multiplying a fixed element $x\in G$
by elements of the $``\cdot"$ table from the left. If the tables
obtained for the operation $``\circ"$ and $``\diamondsuit"$ coincide
for all $x\in G,$ then \ref{eq:b-1} holds, and the AG-groupoid is
an RAD-AG-groupoid.

\noindent \begin{example} Check the following AG-groupoid $G=\{1,2,3\}$
for RAD-AG-groupoid.\end{example}

\noindent \begin{center}
\begin{tabular}{|c|ccc|}
\hline 
$\cdot$ & 1 & 2 & 3\tabularnewline
\hline 
1 & 1 & 1 & 1\tabularnewline
2 & 1 & 1 & 1\tabularnewline
3 & 2 & 2 & 2\tabularnewline
\hline 
\end{tabular}
\par\end{center}

Extend the above table in the way as described above we get the following: 

\noindent \begin{center}
\begin{tabular}{|c|ccc|ccccccccc}
\hline 
$\cdot$ & 1 & 2 & 3 & 1 & 1 & \multicolumn{1}{c|}{1} & 1 & 1 & \multicolumn{1}{c|}{1} & 2 & 2 & \multicolumn{1}{c|}{2}\tabularnewline
\hline 
1 & 1 & 1 & 1 & 1 & 1 & \multicolumn{1}{c|}{1} & 1 & 1 & \multicolumn{1}{c|}{1} & 1 & 1 & \multicolumn{1}{c|}{1}\tabularnewline
2 & 1 & 1 & 1 & 1 & 1 & \multicolumn{1}{c|}{1} & 1 & 1 & \multicolumn{1}{c|}{1} & 1 & 1 & \multicolumn{1}{c|}{1}\tabularnewline
3 & 2 & 2 & 2 & 1 & 1 & \multicolumn{1}{c|}{1} & 1 & 1 & \multicolumn{1}{c|}{1} & 1 & 1 & \multicolumn{1}{c|}{1}\tabularnewline
\hline 
 & 1 & 1 & 1 &  &  &  &  &  &  &  &  & \tabularnewline
1 & 1 & 1 & 1 &  &  &  &  &  &  &  &  & \tabularnewline
 & 1 & 1 & 1 &  &  &  &  &  &  &  &  & \tabularnewline
\cline{1-4} 
 & 1 & 1 & 1 &  &  &  &  &  &  &  &  & \tabularnewline
2 & 1 & 1 & 1 &  &  &  &  &  &  &  &  & \tabularnewline
 & 1 & 1 & 1 &  &  &  &  &  &  &  &  & \tabularnewline
\cline{1-4} 
 & 1 & 1 & 1 &  &  &  &  &  &  &  &  & \tabularnewline
3 & 1 & 1 & 1 &  &  &  &  &  &  &  &  & \tabularnewline
 & 1 & 1 & 1 &  &  &  &  &  &  &  &  & \tabularnewline
\cline{1-4} 
\end{tabular}
\par\end{center}

It is clear from the extended table that the downward tables for the
operation $``\diamondsuit"$ and the tables on the right for the operation
$``\circ"$ coincide so $G$ is an RAD-AG-groupoid.

\begin {theorem}Every RAD-AG-groupoid is right distributive AG-groupoid.\end {theorem}

\begin {proof}Let $S$ be an RAD-AG-groupoid, and let $a,b,c\in S.$
Then by the Identity \ref{eq:b-1} and medial law, we get 
\begin{eqnarray*}
ab\cdot c & = & ca\cdot bc=(bc\cdot c)(a\cdot bc)\\  & = &(bc\cdot a)(c\cdot bc)=ac\cdot bc\\ \Rightarrow ab\cdot c & = & ac\cdot bc.
\end{eqnarray*}
Hence $S$ is right distributive AG-groupoid.\end {proof}

The converse of the previous thorem is not valid. The following example
show that not every RD-AG-groupoid is an RAD-AG-groupoid.

\begin {example}Let $S=\{a,b,c,d\}$ with the following operation.
Then one can easily verify that S is an RD-AG-groupoid, but is not
an RAD-AG-groupoid.

\begin{center}
\begin{tabular}{c|cccc}
{*} & a & b & c & d\tabularnewline
\hline 
a & a & c & d & b\tabularnewline
b & d & b & a & c\tabularnewline
c & b & d & c & a\tabularnewline
d & c & a & b & d\tabularnewline
\end{tabular}
\par\end{center}

\end {example}

\section{\textbf{Abelian Distributive AG-groupoid}}

\begin {definition}An AG-groupoid $S$ is called abelian distributive
AG-groupoid denoted by AD-AG-groupoid if it is both RAD and LAD AG-groupoids.\end {definition}

The following result confirms that non-associative AD-AG-groupoids
do not exist.

\begin {theorem}Every AD-AG-groupoid is a semigroup.\end {theorem}

\begin {proof}Let $S$ be an AD-AG-groupoid, and let $a,b,c\in S.$
Then by identities (\ref{eq:a-1},\ref{eq:b-1}, \ref{eq:a}), and
medial law we get, 
\begin{eqnarray*}
a\cdot bc & = & ab\cdot ca=ac\cdot ba=cb\cdot a=ab\cdot c\Rightarrow a\cdot bc=ab\cdot c.
\end{eqnarray*}
Hence $S$ is a semigroup.\end {proof}

\end{document}